\documentclass[12pt]{amsart}
\usepackage{amsthm}
\usepackage{amsfonts}
\usepackage{amsmath}
\usepackage{amssymb}
\usepackage[all]{xy}

\newtheorem{thm}{Theorem}[section]
\newtheorem{lemma}[thm]{Lemma}
\newtheorem{prop}[thm]{Proposition}

\newtheorem{cor}[thm]{Corollary}

\theoremstyle{definition}
\newtheorem{defn}[thm]{Definition}

\newcommand{\ve}{\varepsilon}

\newcommand{\R}{\mathbb R}

\newcommand{\C}{\mathbb C}
\newcommand{\Hil}{\mathbb H}

\newcommand{\Z}{\mathbb Z}

\newcommand{\K}{\mc{K}}

\def\so{_{\scriptscriptstyle O}}
\def\su{_{\scriptscriptstyle U}}
\def\st{_{\scriptscriptstyle T}}

\newcommand{\crt}{^{\scriptscriptstyle {\it CRT}}}

\newcommand{\ct}{{\it CRT}}

\newcommand{\ftn}[3]{ #1 : #2 \rightarrow #3 }
\newcommand{\mc}[1]{\mathcal{#1}}
\newcommand{\multialg}[1]{\mathcal{M}(#1)}

\newcommand{\sm}[4]{       \bigl( \begin{smallmatrix} 
					{#1} & {#2} \\ {#3} & {#4}
                   		    \end{smallmatrix} \bigr)         }

\newcommand{\catr}{\textbf{C*$\R$-Alg}}

\newcommand{\catabgp}{\textbf{Ab}}
\newcommand{\catkk}{\textbf{KK}}

\begin{document}
	\title{Axiomatic $KK$-theory for Real C*-algebras}
	\author{Jeffrey L. Boersema}
	\address{Department of Mathematics \\
	Seattle University \\
	Seattle, WA 98133 USA}
	\email{boersema@seattleu.edu}
	\author{Efren Ruiz}
        \address{Department of Mathematics \\
        University of Hawaii Hilo \\
        200 W. Kawili St. \\
        Hilo, Hawaii 96766 USA}
        \email{ruize@hawaii.edu}
 
        \date{\today}
	

	
	\begin{abstract}
We establish axiomatic characterizations of $K$-theory and $KK$-theory for real C*-algebras.  In particular, let $F$ be an abelian group-valued functor on separable real C*-algebras.  We prove that if $F$ is homotopy invariant, stable, and split exact, then $F$ factors through the category $KK$.    Also, if $F$ is homotopy invariant, stable, half exact, continuous, and satisfies an appropriate dimension axiom, then there is a natural isomorphism $K(A) \rightarrow F(A)$ for a large class of  separable real C*-algebras $A$.  Furthermore, we prove that a natural transformation $F(A) \rightarrow G(A)$ of homotopy invariant, stable, half-exact functors which is an isomorphism when $A$ is complex is necessarily an isomorphism when $A$ is real.
	\end{abstract}
	
        \maketitle

\section{Introduction}

In this paper we establish axiomatic characterizations of $K$-theory and $KK$-theory for real C*-algebras, following that established in the complex case by \cite{higson87}.  These results and techniques will be used in a forthcoming paper to prove a classification of real simple purely infinite C*-algebras along the lines of that in the complex case proven by \cite{phillips00} and \cite{kirchberg94}.

Let $F$ be an abelian group-valued functor on separable real C*-algebras.  We prove that if $F$ is homotopy invariant, stable, and split exact, then $F$ factors through the category $KK$.  Also, if $F$ is homotopy invariant, stable, half exact, continuous, and satisfies the dimension axiom, then there is a natural isomorphism $K(A) \rightarrow F(A)$ for a large class of  separable real C*-algebras $A$.  These results are proven along the same lines as the proofs in \cite{higson87}.  

If $F$ is homotopy invariant, stable, and split exact then we use $F$ to construct a \ct-module $F\crt(A)$  (see \cite{bousfield90}, \cite{boersema02}, and \cite{boersema04}) for any separable real C*-algebra.
This \ct-module structure allows us to prove that a natural transformation $F(A) \rightarrow G(A)$ of homotopy invariant, stable, half-exact functors must be an isomorphism if it is an isomorphism in the special case that $A$ is complex.

These main results are all in Section 3.  In Section 2 we re-develop the essential preliminaries concerning $KK$-theory in the real case, including the equivalence of the two principal models of $KK$-theory:  the Kasparov bimodule construction and the Fredholm picture.
        
\section{The Fredholm Picture of $KK$-Theory}

We take the following definition from Section~2.3 of \cite{schroderbook} to be the standard definition of $KK$-theory for real C*-algebras.  It is essentially the same as that in \cite{kasparov81} where it was simultaneously developed for both real and complex C*-algebras.  

\begin{defn}\label{def:kbimod}
Let $A$ be a graded separable real C*-algebra and $B$ be a real C*-algebra with a countable approximate identity.
\begin{itemize}
\item[(i)]  A Kasparov ($A$-$B$)-bimodule is a triple $( E , \phi , T )$ where $E$ is a countably generated graded Hilbert $B$-module, $\ftn{ \phi }{ A }{ \mc{L} ( E ) }$ is a graded $*$-homomorphism, and $T$ is an element of $\mc{L} ( E )$ of degree 1 such that
\begin{equation*}
( T - T^{*} ) \phi ( a ), ( T^{2} - 1 ) \phi ( a ) , \ \mathrm{and} \ [ T , \phi ( a ) ] 
\end{equation*}
lie in $\mc{K}(E)$ for all $a \in A$.

\item[(ii)] Two triples $( E_{i} , \phi_{i} , T_{i} )$ are \emph{unitarily equivalent} if there is a unitary $U$ in $\mc{L} ( E_{0} , E_{1} )$, of degree zero, intertwining the $\phi_{i}$ and $T_{i}$. 

\item[(iii)] If $(E, \phi, T)$ is a Kasparov ($A$-$B$)-bimodule and $\beta \colon B \rightarrow B'$ is a homomorphism of C*-algebras, then the pushed-forward Kasparov ($A$-$B'$)-bimodule is defined by $$\beta_*(E,\phi,T) = (E \hat{\otimes}_\beta B', \phi \hat{\otimes} 1, T \hat{\otimes} 1) \; .$$

\item[(iv)]  Two Kasparov ($A$-$B$)-bimodules $(E_{i} , \phi_{i} , T_{i} )$ for $i = 0,1$ are \emph{homotopic} if there is a Kasparov bimodule ($A$-$B { \otimes } C[0,1])$, say $(E, \phi , T )$, such that $(\ve_i)_* (E, \phi, T)$ and $(E_i, \phi_i, T_i)$ are unitarily equivalent for $i = 0,1$, where $\ve_{i}$ denotes the evaluation map. 

\item[(v)]  A triple $(E, \phi , T )$ is \emph{degenerate} if the elements \begin{equation*}
( T - T^{*} ) \phi ( a ), ( T^{2} - 1 ) \phi ( a ) , \ \mathrm{and} \ [ T , \phi ( a ) ]
\end{equation*}
are zero for all $a \in A$.  By Proposition~2.3.3 of \cite{schroderbook}, degenerate bimodules are homotopic to trivial bimodules.
    
\item[(vi)] $KK( A , B )$ is defined to be the set of homotopy equivalence classes of Kasparov ($A$-$B$)-bimodules.  
\end{itemize}
\end{defn}     

The following theorem summarizes the principal properties of $KK$-theory for real C*-algebras from Chapter~2 of \cite{schroderbook}.

\begin{prop} \label{kkproperties}
$KK(A,B)$ is an abelian group for separable $A$ and $\sigma$-unital $B$.
As a functor on separable real C*-algebras
(contravariant in the first argument and covariant in the second argument), it is homotopy invariant, stable, and has split exact sequences in both arguments.  Furthermore, there is a natural associate pairing (the intersection product)
$$\otimes_C \colon KK(A, C \otimes B) \otimes KK(C \otimes A', B') \rightarrow KK(A \otimes A', B \otimes B')  \; .$$
\end{prop}

We now turn to the Fredholm picture of $KK$-theory, which was developed in \cite{higson87} with only the situation of complex C*-algebras in mind.  However, the approach goes through the same for real C*-algebras, as follows.

\begin{defn}
Let $A$ and $B$ be real separable C*-algebras.
\begin{itemize}
\item[(i)]    A \emph{$KK ( A, B)$}-cycle is a triple $( \phi_{+} , \phi_{-} , U )$, where 
$\ftn{ \phi_{ \pm } }{ A }{ \mc{M}({ \mc{K} \otimes B }) }$ are $*$-homomorphisms, and $U$ is an element of $\mc{M}({ \mc{K} \otimes B })$ such that
\begin{equation*}
U \phi_{+} ( a ) - \phi_{-} ( a ) U, \ \phi_{+} ( a ) ( U^{*} U - 1 ), \ \mathrm{and} \ \phi_{-} ( a ) ( U U^{*} - 1 ) 
\end{equation*}
lie in $\mc{K} \otimes B$ for all $a \in A$.  

\item[(ii)]  Two $KK ( A , B )$-cycles $( \phi_{+}^{i} , \phi_{-}^{i} , U^{i} )$ are \emph{homotopic} if there is a $KK ( A , B \otimes C[ 0, 1 ]  )$-cycle $(\phi_{+} , \phi_{-} , U )$ such that $( \varepsilon_{i} \phi_{+}, \varepsilon_{i} \phi_{-} , \varepsilon_{i} ( U ) ) = ( \phi_{+}^{i} , \phi_{-}^{i} , U^{i} )$, where 
$\ftn{ \varepsilon_{i} } { \mc{M}( \mc{K} \otimes B \otimes C [ 0 , 1] ) }{ \mc{M}( \mc{K} \otimes B ) }$ is induced by evaluation at $i$.

\item[(iii)]  A $KK ( A , B )$-cycle $(\psi_{+} , \psi_{-} , V )$ is \emph{degenerate} if the elements 
\begin{equation*}
V \psi_{+} ( a ) - \psi_{-} ( a ) V, \ \psi_{+} ( a ) ( V^{*} V - 1 ), \ \mathrm{and} \ \psi_{-} ( a ) ( V V^{*} - 1 ) 
\end{equation*}
are zero for all $a \in A$.     

\item[(iv)]  The sum $( \phi_{+} , \phi_{-} , U ) \oplus ( \psi_{+} , \psi_{-} ,V )$ of two $KK ( A , B )$-cycles is the $KK ( A , B )$-cycle 
\begin{equation*}
\left( \left( 
\begin{matrix}
\phi_{+} & 0 \\
0 & \psi_{+}
\end{matrix}
\right), 
\left( 
\begin{matrix}
\phi_{-}  & 0 \\
0 & \psi_{-}  
\end{matrix}
\right),
\left( 
\begin{matrix}
U & 0 \\
0 & V
\end{matrix}
\right)
\right)
\end{equation*}
where the algebra $M_2 ( \mc{M}( \mc{K} \otimes B ) )$ is identified with $\mc{M}( \mc{K} \otimes B )$ by means of some isomorphism $M_2  ( \mc{K} ) \cong \mc{K}$, which is unique up to homotopy by Section~1.17 of \cite{kasparov81}.

\item[(v)]  Two cycles $( \phi_{+}^{i}, \phi_{-}^{i} , U^{i} )$ are said to be \emph{equivalent} if there exist degenerate cycles $( \psi_{+}^{i} , \psi_{-}^{i} , V^{i} )$ such that 
\begin{equation*}
( \phi_{+}^{0} , \phi_{-}^{0} , U^{0} ) \oplus ( \psi_{+}^{0}, \psi_{-}^{0}, V^{0} ) \quad \mathrm{and} \quad ( \phi_{+}^{1} , \phi_{-}^{1} , U^{1} ) \oplus ( \psi_{+}^{1}, \psi_{-}^{1}, V^{1} )
\end{equation*}  
are homotopic.  

\item[(vi)] $\textbf{KK} ( A , B )$ is defined to be the set of equivalence classes of $KK ( A , B )$-cycles.

\end{itemize}
\end{defn}

The following lemma is the real version of Lemma~2.3 of \cite{higson87}.

\begin{lemma}  \label{lem:fhgrp}
$\textbf{KK}( A , B )$ is an abelian group, for separable real C*-algebras $A$ and $B$.  As a functor it is contravariant in the first argument and covariant in the second argument.
\end{lemma}

\begin{proof}
For the first statement, we show that a cycle $(\phi_+, \phi_-, U)$ has inverse $(\phi_-, \phi_+, U^*)$.  Indeed,
the sum
\begin{equation*}
\left( 
\left(
\begin{matrix}
\phi_{+} & 0 \\
0 &  \phi_{-} 
\end{matrix}
\right),
\left( 
\begin{matrix}
\phi_{-} & 0 \\
0 &  \phi_{+} 
\end{matrix}
\right),
\left( 
\begin{matrix}
U & 0 \\
0 & U^*
\end{matrix}
\right)
\right)
\end{equation*} 
is homotopic to a degenerate cycle via the operator homotopy
\begin{equation*}
W_{t} = 
\left(
\begin{matrix}
\cos ( t ) U & - \sin(t)  \\
\sin ( t )  & \cos ( t ) U^* 
\end{matrix}
\right), \quad t \in \left[ 0 , \frac{ \pi }{ 2 } \right] \; .
\end{equation*}

The functoriality is established as in \cite{higson87} in Sections~2.4 through 2.7.
\end{proof}

The next proposition establishes the isomorphism between the two pictures of $KK$-theory.  First we review some preliminaries regarding graded C*-algebras and Hilbert modules.

If $B$ is a real C*-algebra, then the \emph{standard even grading} on $M_2 ( B )$ is obtained by setting $M_2 ( B )^{(0)}$ to be the set of diagonal matrices and $M_2( B )^{(1)}$ the set of matrices with zero diagonal.  The \emph{standard even grading} on $\K \otimes B$ is obtained by choosing an isomorphism $\K \otimes B  \cong M_2 ( \K \otimes B )$.  This in turn induces a canonical (modulo a unitary automorphism) grading on $\multialg{ \K \otimes B }$.

Let $\Hil_B$ be the Hilbert $B$-module consisting of all sequences $\{b_n\}_{n = 1}^\infty$ in $B$ such that $\sum_{n = 1}^\infty b_n^* b_n$ converges.  Giving $B$ the trivial grading (that is $B^{(0)} = B$ and $B^{(1)} = \{0\}$), let $\hat{\Hil}_B = \Hil_B \oplus \Hil_B$ be the graded Hilbert $B$-module with $\hat{\Hil}_B^{(0)} = \Hil_B \oplus 0$ and $\hat{\Hil}_B^{(1)} = 0 \oplus \Hil_B$.  Then the induced grading on $\mathcal{L}(\hat{ \Hil}_B)$ is identical with the standard even grading of $M_2 ( \multialg{ \K \otimes B } )$.  Under the isomorphism   $M_2 ( \multialg{ \K \otimes B } ) \cong  \multialg{ \K \otimes B } $, this grading coincides with the one described in the previous paragraph.

\begin{thm}\label{prop:eqdef}
Let $A$ and $B$ be a real separable C*-algebras.  Then $KK ( A , B )$ is isomorphic to $\textbf{KK} ( A , B )$.
\end{thm}        

\begin{proof}
We give $A$ and $B$ the trivial grading and $\multialg{ \K \otimes B }$ is given the standard even grading described above.

For a $KK ( A , B )$-cycle $x = ( \phi_{+} , \phi_{-} , U )$ we define 
\begin{equation*}
\alpha ( x ) = \left( \hat{ \Hil }_{ B } , \left( \begin{matrix} \phi_{+} & 0 \\ 0 & \phi_{-} \end{matrix} \right)  ,  \left( \begin{matrix} 0 & U^{*} \\  U & 0 \end{matrix} \right) \right) \;.
\end{equation*}
It is readily verified that $\alpha(x)$ is a Kasparov ($A$-$B$) bimodule.  Furthermore, for $x = ( \phi_{+} , \phi_{-} , U )$ and $y =  ( \psi_{+} , \psi_{-} , V )$ it is easy to see that
$\alpha(x+y)$ and $\alpha(x) + \alpha(y)$ are unitarily equivalent via a degree 0 unitary.

We must show that $\alpha$ induces a well-defined homomoprhism 
$$\overline{\alpha} \colon  \textbf{KK} ( A , B ) \rightarrow KK(A, B) \; .$$
Note first that $\alpha$ sends degenerate elements to degenerate elements.  Next, suppose $( \phi_{+} , \phi_{-} , U )$ is a $KK ( A , B \otimes C[ 0 , 1 ] )$-cycle implementing a homotopy between $x = ( \phi_{+}^{0} , \phi_{-}^{0} , U_{0} )$ and $y = ( \phi_{+}^{1} , \phi_{-}^{1} , U_{1} )$.  That is, $\varepsilon_{i} ( \phi_{+} ) = \phi_{+}^{i}$, $\varepsilon_{i} ( \phi ) = \phi_{-}^{i}$, and $\varepsilon_{i} ( U ) = U_{i}$;  where $\ftn{ \varepsilon_{t} }{ \multialg{ \K \otimes B \otimes C[0,1] } }{ \multialg{ \K \otimes B } }$ is the map induced by the evaluation map $\varepsilon_{t}$ at $t$.
Consider the Kasparov  ($A$-$B \otimes C[ 0 , 1 ]$)-bimodule
\begin{equation*}
z = \left( \hat{ \Hil }_{ B \otimes C[0,1]  } , \left( \begin{matrix}  \phi_{+} & 0 \\ 0 & \phi_{-} \end{matrix} \right) , \left( \begin{matrix} 0 & U^{*} \\ U & 0 \end{matrix} \right) \right).
\end{equation*}
Since $ 
 \hat{ \Hil }_{ B \otimes C[0,1] } \hat{ \otimes }_{ \epsilon_{i} } B \cong \hat{ \Hil }_{ B }$ it follows that
\begin{equation*}
(\varepsilon_0)_*(z) = \left( \hat{ \Hil }_{ B } , \left( \begin{matrix}  \phi_{+}^{0} & 0 \\ 0 & \phi_{-}^{0} \end{matrix} \right) , \left( \begin{matrix} 0 & U_{0}^{*} \\ U_{0} & 0 \end{matrix} \right) \right)
\end{equation*}
and
\begin{equation*}
(\varepsilon_1)_*(z) =\left( \hat{ \Hil }_{ B } , \left( \begin{matrix}  \phi_{+}^{1} & 0 \\ 0 & \phi_{-}^{1} \end{matrix} \right) , \left( \begin{matrix} 0 & U_{1}^{*} \\ U_{1} & 0 \end{matrix} \right) \right)
\end{equation*}
Therefore $\alpha(x)$ and $\alpha(y)$ are homotopic.  Hence $\overline{\alpha}$ is well-defined.  

To show that $\overline{ \alpha }$ is surjective, let $y = (E , \phi , T )$ be a Kasparov ($A$,$B$)-bimodule.  By Proposition~2.3.5 of \cite{schroderbook}, we may assume that $E \cong \hat{ \Hil}_B$ and that $T = T^*$.  Thus, with respect to the graded isomorphism $\mathcal{L}({\hat{ \Hil}_B}) \cong M_2(\multialg{ \otimes B})$,
we can write
$$
\phi = \left( \begin{matrix}  \phi_{+} & 0 \\ 0 & \phi_{-} \end{matrix} \right) 
\qquad \text{and} \qquad
T  = \left( \begin{matrix} 0 & U^{*} \\ U& 0 \end{matrix} \right)
$$
where $\phi_+$ and $\phi_-$ are homomorphisms from $A$ to $\multialg{ \K \otimes B}$ and $U$ is an element of $\multialg{ \K \otimes B}$.  Then $y = \overline{\alpha}(x)$ where $x = (\phi_+, \phi_-, U)$.

Finally, we show that $\overline{\alpha}$ is injective.  Suppose that $\overline{\alpha}(x) = \overline{\alpha}(y)$ where $x = (\phi_+, \phi_-, U)$ and $y = (\psi_+, \psi_-, V)$.  Then there is a Kasparov $(A,B \otimes C[ 0 , 1 ]  )$-bimodule $z = (E, \phi, T)$ such that 
$$
(\ve_0)_*(z) \cong  \left( \hat{\Hil}_B, \sm{\phi_+ }{0}{0}{\phi_-}, \sm{0}{U^*}{U}{0} \right) $$
and
$$
(\ve_1)_*(z) \cong  \left( \hat{\Hil}_B, \sm{\psi_+ }{0}{0}{\psi_-}, \sm{0}{V^*}{V}{0} \right)  \; .
$$

As above, we may assume that $E = \hat{\Hil}_{B \otimes C[0,1]}$ and that $T = T^*$.  Then $z$ has the form
$$z = \left( \hat{\Hil}_{B \otimes C[0,1]}, \sm{\theta_+}{0}{0}{\theta_-}, \sm{0}{W^*}{W}{0} \right) $$
where $\theta_+$ and $\theta_-$ are homomorphisms to $\multialg{ \K \otimes B \otimes C[0,1]}$ and $W$ is an element of $\multialg{ \K \otimes B \otimes C[0,1]}$.  Then $(\theta_+, \theta_-, W)$ is a Kasparov $(A, B \otimes C[0,1])$-bimodule implementing a homotopy between $x$ and $y$.

\end{proof}
        
\section{The Universal Property of $KK$-Theory } 

Let $F$ be a functor from the category {\catr}  of separable real C*-algebras to the cateogry {\catabgp} of abelian groups.  We say that $F$ is
\begin{enumerate}
\item[(i)] {\it homotopy invariant} if $(\alpha_1)_* = (\alpha_2)_*$ whenever $\alpha_1$ and $\alpha_2$ are homotopic homomorphisms on the level of real C*-algebras.
\item[(ii)] {\it stable} if $e_* \colon F(A) \rightarrow F(\mathcal{K} \otimes A)$ is an isomorphism for the inclusion $e \colon A \hookrightarrow \mathcal{K} \otimes A$ defined via any rank one projection.
\item[(iii)] {\it split exact} if any split exact sequence of separable C*-algebras
$$0 \rightarrow A \rightarrow B \rightarrow C \rightarrow 0$$
induces a split exact sequence 
$$0 \rightarrow F(A) \rightarrow F(B) \rightarrow F(C) \rightarrow 0 \; .$$
\item[(iv)] {\it half exact} if any short exact sequence of separable C*-algebras
$$0 \rightarrow A \rightarrow B \rightarrow C \rightarrow 0$$
induces an exact sequence 
$$F(A) \rightarrow F(B) \rightarrow F(C) \; .$$
\end{enumerate}

In what follows we will see that if $F$ is homotopy invariant and half exact, then it is split exact.

The following theorem is the version for real C*-algebras of Theorem~3.7 of \cite{higson87} and Theorem~22.3.1 of \cite{blackadarbook}.

\begin{thm} \label{thm:univprop1}
Let $F$ be a functor from {\catr} to {\catabgp} that is homotopy invariant, stable, and split exact.  Then there is a unique natural pairing $\alpha \colon F(A) \otimes KK(A,B) \rightarrow F(B)$ such that $\alpha(x \otimes 1_A) = x$ for all $x \in F(A)$ and where $1_A \in KK(A,A)$ is the class represented by the identity homomorphism.

Furthermore, the pairing respects the intersection product on $KK$-theory in the sense that
$$\alpha(\alpha(x \otimes y) \otimes z) = \alpha(x \otimes ( y \otimes_B z)) \colon F(A) \otimes KK(A,B) \otimes KK(B,C) \rightarrow F(C) \; .$$
\end{thm}

\begin{proof}
Let $\Phi \in KK(A,B)$.  Using Theorem~\ref{prop:eqdef} we represent $\Phi$ with a $KK(A,B)$ cycle and as in Lemma~3.6 of \cite{higson87}, we may assume that this cycle has the form $(\phi_+, \phi_-, 1)$.  We use the same construction as in Definitions~3.3 and 3.4 in \cite{higson87}.  In that setting $F$ is assumed to be a functor from separable complex C*-algebras, but it goes through the same for functors from separable real C*-algebras to any abelian category.  This construction produces a homomorphism $\Phi_* \colon F(A) \rightarrow F(B)$ and we then define $\alpha(x \otimes \Phi) = \Phi_*(x)$.  The proof of Theorems~3.7 and 3.5 of \cite{higson87} carry over in the real case to show that $\alpha$ is natural, is well-defined, satisfies $\alpha(x \otimes 1_A) = x$, and is unique.

That $\alpha$ respects the Kasparov product follows from the uniqueness statement.
\end{proof}

For any real C*-algebra $A$ we define $SA = C_0(\R, A)$ and $S^{-1}A = \{ f \in C_0(\R, \C \otimes A) \mid f(-x) = \overline{f(x)} \}$.  For any functor $F$ on {\catr} and any integer $n$, we define $F_n(A) = F(S^n(A))$ where
$$S^n(A) = \begin{cases} S^n(A) & n > 0 \\
					A & n = 0 \\
					(S^{-1})^{-n} & n < 0 \; .
			\end{cases}
$$

\begin{cor} \label{period}
Let $F$ be a functor from {\catr} to {\catabgp} that is homotopy invariant, stable, and split exact.  Then $F_*(A)$ has the structure of a graded module over the ring $K_*(\R)$.  In particular, $F(S^8 A) \cong F(A)$ and $F(S^{-1} SA) \cong F(A)$.
\end{cor}

\begin{proof}
For all separable $A$ and $\sigma$-unital $B$, the pairing of Proposition~\ref{kkproperties} gives $KK_*(A,B)$ the structure of a module over $KK_*(\R, \R)$.  Taking $A = B$, we define a graded ring homomorphism $\beta$ from $K_*(\R) \cong KK_*(\R, \R)$ to $KK_*(A,A)$ by multiplication by $1_A \in KK(A,A)$.

Then for any $x \in F_m(A)$ and $y \in K_n(\R)$ we define 
$x \cdot y = \alpha(x \otimes \beta(y)) \in F_{n+m}(A)$.
The second statement follows from the $KK$-equivalence between $\R$ and $S^8 \R$, and that between $\R$ and $S^{-1} S \R$ from Section~1.4 of \cite{boersema02}. 
\end{proof}

For all integers $n$ and $m$ there is a $KK$-equivalence between $S^n S^m A$ and $S^{n+m} A$, so it follows that the pairing Theorem~\ref{thm:univprop1} extends to a well-defined graded pairing 
$$\alpha \colon F_*(A) \otimes KK_*(A,B) \rightarrow F_*(B) \; .$$

Let {\catkk} be the category whose objects are separable real C*-algebras and the set of morphisms from $A$ to $B$ is $KK(A , B )$.  There is a canonical functor $KK$ from {\catr} to {\catkk} that takes an object $A$ to itself and which takes a C*-homomorphism $f \colon A \rightarrow B$ to the corresponding element $[f] \in KK(A,B)$.    

\begin{cor}\label{thm:univprop2}
Let $F$ be a functor from {\catr} to {\catabgp} that is homotopy invariant, stable, and split exact. 
Then there exists a unique functor $\ftn{ \hat{F} }{ \catkk }{ \textbf{A} }$ such that $\hat{F} \circ KK = F$.
\end{cor}

\begin{proof}
The functor $\hat{F}$ takes an object $A$ in {\catkk} to $F(A)$ in {\catabgp} and takes a morphism $y \in KK(A,B)$ to the homomoprhism $F(A) \rightarrow F(B)$ defined by $y \mapsto \alpha(x \otimes y)$.  The composition $\hat{F} \circ KK = F$ clearly holds on the level of objects.  On the level of morphisms we must verify the formula $\alpha(x \otimes [f]) = f_*(x)$ for $f \colon A \rightarrow B$ and $x \in F(A)$.  This formula follows by the naturality of the pairing $\alpha$, the formula $\alpha(x \otimes 1_A) = x$, and the formula $f_*(1_A) = [f] \in KK(A,B)$ which is verified as in Section~2.8 of \cite{higson87}.
\end{proof}

\begin{prop} \label{les1}
Let $F$ be a functor from {\catr} to {\catabgp} that is homotopy invariant and half exact.  Then for any short exact sequence 
$$0 \rightarrow A \xrightarrow{f} B \xrightarrow{g} C \rightarrow 0$$ 
there is a natural boundary map $\partial \colon F(SC) \rightarrow F(A)$ that fits into a (half-infinite) long exact sequence
$$\dots \rightarrow F(SB) \xrightarrow{g_*} F(SC) \xrightarrow{\partial} 
	 F(A) \xrightarrow{f_*} F(B) \xrightarrow{g_*} F(C) \; .$$
\end{prop}

\begin{proof}
Use the mapping cone construction as in Section~21.4 of \cite{blackadarbook}.
\end{proof}

\begin{cor} \label{splitexact}
A functor $F$ from {\catr} to {\catabgp} that is homotopy invariant and half exact is also split exact.
\end{cor}

\begin{proof}
The splitting implies that $g_*$ is surjective.  Thus in the sequence of Proposition~\ref{les1}, $\partial = 0$ and $f_*$ is injective.
\end{proof}

\begin{prop}
Let $F$ be a functor from {\catr} to {\catabgp} that is homotopy invariant, stable,  and half exact.  Then for any short exact sequence 
$$0 \rightarrow A \xrightarrow{f} B \xrightarrow{g} C \rightarrow 0$$ 
there is a natural long exact sequence (with 24 distinct terms)
$$\dots \rightarrow F_{n+1}(C) \xrightarrow{\partial} F_n(A) \xrightarrow{f_*} F_n(B) \xrightarrow{g_*} F_n(C) \xrightarrow{\partial} F_{n-1}(A) \rightarrow \dots \; .$$
\end{prop}

\begin{proof}
From Corollary~\ref{splitexact} and Corollary~\ref{period}, $F$ is periodic; so Proposition~\ref{les1} gives the long exact sequence. 
\end{proof}

We say that a homotopy invariant, stable, half-exact functor $F$ from {\catr} to the category {\catabgp} of abelian groups 
\begin{enumerate}
\item[(v)] satisfies the {\it dimension axiom} if there is an isomorphism $F_*(\R) \cong K_*(\R)$ as graded modules over $K_*(\R)$
\item[(vi)] is {\it continuous} if for any direct sequence of real C*-algebras $(A_n, \phi_n)$, the natural homomorphism
$$\lim_{n \to \infty} F_*(A_n) \rightarrow F_*(\lim_{n \to \infty} (A_n)) \; $$
is an isomorphism.
\end{enumerate}

\begin{thm}
Let $F$ be a functor from {\catr} to {\catabgp} that is homotopy invariant, stable, half exact and satisfies the dimension axiom.  Then there is a natural transformation $\beta \colon K_n(A) \rightarrow F_n(A)$.  If $F$ is also continuous, then $\beta$ is an isomorphism for all real C*-algebras in the smallest class of separable C*-algebras which contains $\R$ and is closed under $KK$-equivalence, countable inductive limits, and the two-out-of-three rule for exact sequences.
\end{thm}

\begin{proof}
Let $z$ be a generator of $F(\R) \cong \Z$ and for $x \in K_n(A) \cong KK(\R, S^nA)$ define 
a $K_*(\R)$-module homomorphism $\beta \colon K_*(A) \rightarrow F_*(A)$ by $\beta(x) = \alpha(z \otimes x)$.  Taking $A = \R$, Theorem~\ref{thm:univprop1} yields that $\beta(1_0) = z$ where $1_0$ is the unit of the ring $K_*(\R) = KK_*(\R, \R)$.  Therefore, $\beta$ is an isomorphism for $A = \R$.  Then bootstrapping arguments show that $\beta$ is an isomorphism for all real C*-algebras in the class described.
\end{proof}



From Section~2.1 of \cite{boersema04} we have distinguished elements
\begin{align*}
c \in KK_{0} ( \R, \C ), &\quad\quad r \in KK_{0} ( \C , \R ) \\
\varepsilon \in KK_{0} ( \R , T ), &\quad\quad \zeta \in KK_{0} ( T, \C ) \\
\psi\su \in KK_{0} ( \C , \C ), &\quad\quad \psi\st \in KK_{0} ( T, T ) \\
\gamma \in KK_{-1} ( \C , T ), &\quad\quad \tau \in KK_{1} ( T, \R) \; .
\end{align*}     
  
For any homotopy invariant, stable, split exact functor $F$ on {\catr}, define the united $F$-theory of a real C*-algebra $A$ to be
$$F\crt(A) = \{ F_*(A), F_*(\C \otimes A), F_*(T \otimes A) \} \; $$
together with the collection of natural homomorphisms \begin{align*}
c_n &\colon F_n(A) \rightarrow F_n(\C \otimes A) \\
r_n&\colon  F_n(\C \otimes A) \rightarrow F_n(A)  \\
\ve_n &\colon F_n(A) \rightarrow F_n(T \otimes A) \\
\zeta_n &\colon F_n(T \otimes A) \rightarrow F_n(\C \otimes A) \\
(\psi\su)_n &\colon F_n(\C \otimes A) \rightarrow F_n(\C \otimes A) \\
(\psi\st)_n &\colon F_n(T \otimes A) \rightarrow F_n(T \otimes A) \\
\gamma_n &\colon F_n(\C \otimes A) \rightarrow F_{n-1}(\C \otimes A) \\
\tau_n &\colon F_n(T \otimes A) \rightarrow F_{n+1}(A) \\
\end{align*}
induced by the elements $c,r, \ve, \zeta, \psi\su, \psi\st, \gamma, \tau$ via the 
pairing of Theorem~\ref{thm:univprop1}.

\begin{prop}\label{thm:crtmod}
Let $F$ be a homotopy invariant, stable, split exact functor from {\catr} to {\catabgp} and let $A$ be a separable real C*-algebra.  Then $F\crt(A)$
is a \ct-module.  Moreover, if in addition $F$ is half exact, then $F\crt ( A )$ is acyclic.
\end{prop}
        
\begin{proof}
To show that $F\crt(A)$ is a \ct-module, 
we must show that the \ct -module relations 
\begin{align*}  
rc &= 2    & \psi\su \beta\su &= -\beta\su \psi\su & \xi &= r \beta\su^2 c \\
cr &= 1 + \psi\su & \psi\st \beta\st &= \beta\st \psi\st 
	&  \omega &= \beta\st \gamma \zeta \\ 
r &= \tau \gamma & \varepsilon \beta\so &= \beta\st^2 \varepsilon 
	& \beta\st \varepsilon \tau 
			&= \varepsilon \tau \beta\st + \eta\st \beta\st    \\
c &= \zeta \varepsilon & \zeta \beta\st &= \beta\su^2 \zeta
	 &  \varepsilon r \zeta &= 1 + \psi\st   \\
(\psi\su)^2 &= 1 & \gamma \beta\su^2 &= \beta\st \gamma       
	&  \gamma c \tau &= 1 - \psi\st \\
(\psi\st)^2 &= 1 & \tau \beta\st^2 &= \beta\so \tau & \tau &= -\tau \psi\st  \\
\psi\st \varepsilon &= \varepsilon & 
	\gamma &= \gamma \psi\su 
	\qquad & \tau \beta\st \varepsilon &= 0 \\
\zeta \gamma &= 0 & \eta\so &= \tau \varepsilon 
	& \varepsilon \xi &= 2 \beta\st \varepsilon \\
\zeta &= \psi\su \zeta & \eta\st &= \gamma \beta\su \zeta 
	& \xi \tau &= 2 \tau \beta\st \; 
\end{align*}
hold among the operations $\{c_n, r_n, \ve_n, \zeta_n, (\psi\su)_n, (\psi\st)_n, \gamma_n, \tau_n\}$ on $F\crt(A)$.  But in the proof of Proposition~2.4 of \cite{boersema04}, it is shown that these relations hold at the level of $KK$-elements.  Therefore, using the associativity of the pairing of Theorem~\ref{thm:univprop1}, the same relations hold among the operations of $F\crt(A)$.

Suppose now that $F$ is also half-exact.  To show that $K\crt(A)$ is acyclic, 
we must show that the sequences 
\begin{equation} \label{CRTequation1}
\dots  \rightarrow F_n(A) \xrightarrow{\eta\so} F_{n+1}(A) \xrightarrow{c} 
F_{n+1}(\C \otimes A) \xrightarrow{r \beta\su^{-1}} F_{n-1}(A) \rightarrow \dots 
\end{equation}
\begin{equation} \label{CRTequation2}
\dots  \rightarrow F_n(A) \xrightarrow{\eta\so^2} F_{n+2}(A) \xrightarrow{\ve} 
F_{n+2}(T \otimes A) \xrightarrow{\tau \beta\st^{-1}} F_{n-1}(A) \rightarrow \dots 
\end{equation}
\begin{equation} \label{CRTequation3}
\dots  \rightarrow F_{n+1}(\C \otimes A) \xrightarrow{\gamma} F_{n}(T \otimes A) \xrightarrow{\zeta}
F_{n}(\C \otimes A) \xrightarrow{1 - \psi\su} F_{n}(\C \otimes A) \rightarrow \dots  \; 
\end{equation}
are exact.  These can be derived from the short exact sequences
$$0 \rightarrow S^{-1} \R \otimes A \rightarrow \R \otimes A \rightarrow \C \otimes A \rightarrow 0  
$$
$$0 \rightarrow S^{-2} \R \otimes A \rightarrow \R \otimes A \rightarrow T \otimes A \rightarrow 0  
$$
$$0 \rightarrow S \C \otimes A \rightarrow T \otimes A \rightarrow \C \otimes A \rightarrow 0  
$$
from Sections~1.2 and 1.4 of \cite{boersema02}.  

Indeed, focusing on the first one for our argument, the short exact sequence
$$0 \rightarrow S^{-1} \R \otimes A \xrightarrow{f} \R \otimes A \xrightarrow{g} \C \otimes A \rightarrow 0  
$$
gives rise to the long exact sequence
$$\dots  \rightarrow F_n(A) \xrightarrow{[f]} F_{n+2}(A) \xrightarrow{[g]} 
F_{n+2}(T \otimes A) \xrightarrow{\partial} F_{n-1}(A) \rightarrow \dots $$
where the homomorphisms are given by the multiplication of the elements $[f] \in KK_2(\R, \R)$, $[g] \in KK (\R, \C)$, and $\partial \in KK_{-1}(\C, \R)$.
In the case of the functor $KK(B,-)$, it was shown in the proof of Proposition~2.4 of \cite{boersema04} that the resulting sequence has the form
$$ \dots \rightarrow KK_{n} ( B , A ) \xrightarrow{\eta\so} 
KK_{ n + 1} ( B , A ) \xrightarrow{c} 
KK_{n+1} ( B , \C \otimes A ) \xrightarrow{r \beta\su^{-1}}  
 \cdots   $$
for all separable real C*-algebras $B$.  It then follows easily that the $KK$-element equalities $[f] = \eta\so$, $[g] = c$, $\partial = r \beta\su^{-1}$ hold.  This proves that Sequence~\ref{CRTequation1} is exact.   Sequences~\ref{CRTequation2} and \ref{CRTequation3} are shown to be exact the same way.

\end{proof}

\begin{thm}
Let $F$ and $G$ be
homotopy invariant, stable, half exact functors from {\catr} to {\catabgp} with a natural transformation $\mu_A \colon F(A) \rightarrow G(A)$.  If $\mu_A$ is an isomorphism for all complex C*-algebras $A$ in {\catr}, then $\mu_A$ is an isomorphism for all real C*-algebras in {\catr}.
\end{thm}

\begin{proof}
Let $A$ be a real separable C*-algebra.  The natural transformation $\mu_A$ induces a homomorphism
$\mu_A\crt \colon F\crt(A) \rightarrow G\crt(A)$ 
of acyclic \ct-modules which is, by hypothesis, an isomorphism on the complex part.  Then the results in Section~2.3 of \cite{bousfield90} imply that $\mu_A\crt$ is an isomorphism.
\end{proof}


\end{document}